\documentclass[reqno, 10.5pt]{amsart}

\usepackage{amsfonts,amsmath,amssymb,amsthm,amscd,enumerate,hyperref,comment}
\usepackage[dvipsnames]{xcolor}

\newtheorem{thm}{Theorem}[section]

\newtheorem{lem}{Lemma}[section]

\theoremstyle{remark}
\newtheorem{rmk}{Remark}[section]

\numberwithin{equation}{section}

\theoremstyle{definition}
\newtheorem{defn}{Definition}[section]

\newcommand{\R}{\ensuremath{\mathbb{R}}}

\newcommand{\rS}{\ensuremath{\mathbb{S}}}

\newcommand{\na}{\nabla}

\newcommand{\vphi}{\varphi}

\newcommand{\tr}{\operatorname{tr}}

\begin{document}

\title[On the classification of 4D steady and expanding Ricci solitons] {A note on the classification of four-dimensional gradient steady and expanding Ricci solitons}
\author[Huai-Dong Cao and Junming Xie]{Huai-Dong Cao and Junming Xie}

\address{Department of Mathematics, Lehigh University, Bethlehem, PA 18015}
\email{huc2@lehigh.edu}

\address{Department of Mathematics, Rutgers University, Piscataway, NJ 08854}
\email{junming.xie@rutgers.edu}

\begin{abstract}

	In this note, we study the classification of four-dimensional complete gradient steady and expanding Ricci solitons. Specifically, under the asymptotically cylindrical (respectively, asymptotically conical) assumption, we classify gradient steady (respectively, expanding) Ricci solitons with half-harmonic Weyl curvature. In addition, we obtain a partial classification of four-dimensional gradient expanding Ricci solitons with half-nonnegative isotropic curvature.
	
\end{abstract}

\maketitle


\section{Introduction}

In this note, building on our recent work \cite{Cao-Xie:23, Cao-Xie:25, Cao-Xie:25b}, we continue the investigation of four-dimensional complete gradient Ricci solitons with special geometry. Our primary focus is on four-dimensional gradient expanding solitons with half-nonnegative isotropic curvature, as well as gradient steady and expanding solitons with half-harmonic Weyl curvature. 

Recall that a Riemannian manifold $(M^n, g)$ is called a {\em gradient Ricci soliton} if there exists
a smooth potential function $f$ on $M^n$ such that the Ricci tensor $Rc$ satisfies the equation
\begin{equation*}
	Rc + \nabla^2 f=\rho g
\end{equation*}
for some constant $\rho \in{\mathbb R}$, where $\na^2 f$ denotes the Hessian of $f$. The Ricci soliton is called expanding if $\rho <0$, steady if $\rho =0$, or shrinking if $\rho>0$. Gradient Ricci solitons play a fundamental role in the study of Hamilton's Ricci flow \cite{Ha:82},  as they generate self-similar solutions and model the formation of singularities
\cite{Ha:95F, Perelman:03, Naber:10, Enders-Mueller-Topping:11}. In particular, steady solitons often arise as Type II singularity models \cite{Ha:95F, Cao:97}, while expanding solitons typically emerge as Type III singularity models \cite{Cao:97, Chen-Zhu:00}. Consequently, their classification remains a central problem in Ricci flow and geometric analysis.  

In the steady case, Hamilton \cite{Ha:88} classified all two-dimensional complete gradient solitons, while the three-dimensional ones were fully classified by the work of Brendle \cite{Brendle:13} and Lai \cite{Lai:24, Lai:25, Lai:25b}. In dimensions $n\geq 4$, classification results have been obtained under various curvature conditions, such as locally conformally flat \cite{Cao-Chen:12, Catino-Mantegazza:11},  harmonic Weyl \cite{Kim:17, LiFJ:25, Kim:25}, and Bach-flat or $D$-flat conditions \cite{CaoCCMM:14, Cao-Yu:21}. Moreover, Brendle \cite{Brendle:14} extended his three-dimensional result to higher dimensions for asymptotically cylindrical steady solitons (see also \cite{Deng-Zhu:20a, Deng-Zhu:20b}). 

In dimension four, it is well known that  the bundle of 2-forms on an oriented four-manifold $(M^4,g)$ admits an orthogonal decomposition, $$\wedge^2(M) = \wedge^{+}(M) \oplus \wedge^{-}(M),$$ into the self-dual and anti-self-dual 2-forms. Accordingly, the Riemann curvature operator $ Rm : \wedge^2(M) \to \wedge^2(M)$, considered as a self-adjoint linear map, admits a block decomposition
\begin{equation} \label{eq:Rmdecomp}
	Rm = 
	\begin{pmatrix}
		A & B \\
		B^t & C
	\end{pmatrix}
	=
	\begin{pmatrix}
		W^+ +\frac{R}{12}I & \mathring{Rc} \\
		\mathring{Rc} & W^- + \frac{R}{12}I
	\end{pmatrix},
\end{equation}
where $W^{\pm}$ denote the self-dual and anti-self-dual parts of the Weyl tensor, and $\mathring{Rc}$ denotes the traceless Ricci tensor. $(M^4,g)$ is said to be {\em half locally conformally flat} if either $W^+=0$ or $W^-=0$. Similarly, $(M^4,g)$ is said to have {\em half-harmonic Weyl curvature} if either $\delta W^+=0$ or $\delta W^-=0$, where $\delta$ denotes the divergence operator.

Half locally conformally flat gradient shrinking and steady solitons were classified by Chen--Wang \cite{Chen-Wang:15}, while asymptotically conical expanding solitons in this class were recently shown to be rotationally symmetric with positive curvature operator in \cite{Cao-Xie:25b}. On the other hand, although four-dimensional gradient shrinking solitons with half-harmonic Weyl curvature were classified by Wu--Wu--Wylie \cite{WWW:18}, the steady case remains open.

Our first main result provides a classification of four-dimensional {\em asymptotically cylindrical}\footnote{``Asymptotically cylindrical" is understood in the sense of Brendle \cite[Definition on p. 191]{Brendle:14}.} gradient steady solitons with half-harmonic Weyl curvature.  

\begin{thm}\label{thm:cylin_half-harmonic}
	Let $(M^4,g,f)$ be a four-dimensional complete, noncompact, asymptotically cylindrical gradient steady Ricci soliton with half-harmonic Weyl curvature. Then, $(M^4,g,f)$ is isometric to the Bryant soliton up to scalings.
\end{thm}

The proof of Theorem \ref{thm:cylin_half-harmonic} builds on computations from \cite{WWW:18} and adapts the maximum principle argument developed in our recent work  \cite{Cao-Xie:25b} on curvature pinching for asymptotically conical expanding solitons. 

Using a similar approach, we also obtain the following classification result for asymptotically conical\footnote{See \cite[Definition 2.1]{Cao-Xie:25b} for the notion of {\em asymptotically conical gradient expanding solitons}.} gradient expanding solitons with half-harmonic Weyl curvature. 

\begin{thm}\label{thm:conical_half-harmonic}
	Let $(M^4,g,f)$ be a four-dimensional complete, noncompact,  asymptotically conical gradient expanding Ricci soliton with half-harmonic Weyl curvature.  If $(M^4,g,f)$ is asymptotic to a non-flat Euclidean cone, then it has positive curvature operator and is rotationally symmetric.
\end{thm}

\begin{rmk}
	By \cite[Equation (A.9)]{Cao-Xie:25b} and \cite[Lemma A.2]{Cao-Xie:25b}, the condition that the asymptotic cone is a non-flat Euclidean cone is equivalent to requiring that it has positive scalar curvature and is half locally conformally flat.  Thus, Theorem~\ref{thm:conical_half-harmonic} slightly strengthens our earlier result \cite[Corollary 1.1]{Cao-Xie:25b}. 
\end{rmk}

We next consider four-dimensional complete noncompact gradient expanding Ricci solitons with half-nonnegative isotropic curvature. Recall that an oriented four-manifold $(M^4,g)$ is said to have {\em positive isotropic curvature} (PIC) if and only if the matrices $A$ and $C$ in \eqref{eq:Rmdecomp} are both 2-positive \cite{Ha:97, Micallef-Wang:93}. Moreover, it is said to have {\em half-positive isotropic curvature} (half-PIC) if and only if either $A$ or $C$ is 2-positive. Analogously, one defines  {\em nonnegative isotropic curvature}/{\em weakly positive isotropic curvature} (WPIC) and  {\em half-nonnegative isotropic curvature} (half-WPIC). For additional background on PIC and WPIC in dimension four, we refer the reader to Micallef--Wang \cite{Micallef-Wang:93} and Hamilton \cite{Ha:97}. 

Four-dimensional gradient shrinking Ricci solitons with PIC or WPIC were classified by Li--Ni--Wang \cite{Li-Ni-Wang:18}. Partial classification results for shrinking solitons with half-WPIC, steady solitons with either WPIC or half-WPIC, and  expanding solitons with  WPIC were obtained recently in our work \cite{Cao-Xie:23, Cao-Xie:25}. In particular, the results for shrinking and steady solitons with half-WPIC in \cite{Cao-Xie:23, Cao-Xie:25} relied on a combination of  curvature pinching estimate \cite[Proposition 3.1]{Cao-Xie:23} (see also \cite{Cho-Li:23}) and Hamilton's strong maximum principle.

However, this approach does not extend directly to expanding solitons with half-WPIC,  since an analogous curvature pinching estimate is not currently available in that setting. Nevertheless, by combining an observation from Micallef--Wang \cite{Micallef-Wang:93} with the strong maximum principle argument in \cite{Cao-Xie:25}, we are able to obtain the following partial classification result in the expanding case.

\begin{thm} \label{thm:half-WPIC-expander}
	Let $(M^4, g, f)$ be a four-dimensional complete noncompact, non-flat, gradient expanding Ricci soliton with half-nonnegative isotropic curvature. Then, either
	\begin{itemize}
		\item[(i)] $(M^4, g, f)$ has half-positive isotropic curvature, or
		
		\smallskip	
		\item[(ii)] $(M^4, g, f)$ is a locally irreducible expanding K\"ahler-Ricci soliton, or
		
		\smallskip	
		\item[(iii)] $(M^4, g, f)$ is isometric to a quotient of either $N^3\times \R$, where $N^3$ is a $3$-dimensional expanding Ricci soliton, or $\Sigma \times \R^2$, or $\Sigma_1 \times \Sigma_2$, where $\Sigma$, $\Sigma_1$, and $\Sigma_2$ are one of the two-dimensional complete gradient expanding Ricci solitons in \cite[Theorem 1]{Bernstein-Mettler:15} and \cite{Ramos:18} {\rm (see Theorem \ref{thm:halfWPIC} for more details)}.
	\end{itemize}
\end{thm}

In view of Theorem~\ref{thm:half-WPIC-expander} and \cite[Theorem 1.3]{Cao-Xie:25}, the classification of expanding and steady  solitons with half-WPIC is essentially reduced to the case of half-PIC. Under additional assumptions on the Ricci tensor and the asymptotic  behavior at infinity, we obtain further classification results; see Theorem~\ref{thm:conical_half-PIC} and Theorem~\ref{thm:cylin_half-PIC} in Section \ref{sec:half-WPIC}.

\medskip
\noindent {\bf Organization of the Paper.} Section \ref{sec:ac_half-harmonic} is devoted to the proofs of Theorems \ref{thm:cylin_half-harmonic} and \ref{thm:conical_half-harmonic}. In  Section \ref{sec:half-WPIC}, we prove Theorem \ref{thm:half-WPIC-expander} and further analyze the half-PIC case under additional assumptions.

\medskip
\noindent {\bf Acknowledgements.} This paper is based upon work supported by a grant from the Institute for Advanced Study School of Mathematics during the Spring 2026 term. The first author was also supported in part by a Simons Fellowship and a grant from the Simons Foundation. The second author would like to thank Prof. Xiaochun Rong for his constant support and encouragement.

\section{Proof of Theorem \ref{thm:cylin_half-harmonic} and Theorem \ref{thm:conical_half-harmonic}} \label{sec:ac_half-harmonic}
In this section, we study four-dimensional gradient steady and expanding solitons with half-harmonic Weyl curvature and  prove Theorems \ref{thm:cylin_half-harmonic} and \ref{thm:conical_half-harmonic}. 

We begin by recalling two useful facts from Wu--Wu--Wylie \cite{WWW:18}.

\begin{lem} [{\cite[Proposition~3.2]{WWW:18}}] \label{lem:WWW-Delta_F}
	Let $(M^4,g,f)$ be a four-dimensional gradient Ricci soliton. Define
	$ F := f - 2\log R$. Then, whenever $|W^{\pm}| \neq 0 $, we have
	\begin{equation*}
		\begin{split}
			\Delta_{F} \left(\frac{|W^{\pm}|}{R}\right) &\geq 
			\frac{1}{2 |W^{\pm}|\,R^{2}}\Big(
			R^{2} |W^{\pm}|^{2}
			- 36R \det W^{\pm}
			+ 4 |W^{\pm}|^{2} \big|\mathring{{Rc}}\big|^{2} \\
			&\quad - R \big\langle (\mathring{{Rc}}\circ \mathring{{Rc}})^{\pm},
			W^{\pm}\big\rangle
			\Big).
		\end{split}
	\end{equation*}
\end{lem}

\begin{lem} [{\cite[Lemma~3.1 \& Remark 3.1]{WWW:18}}] \label{lem:WWW>=0}
	Let $(M^4,g,f)$ be a four-dimensional gradient Ricci soliton with half-harmonic Weyl curvature
	$ \delta W^{\pm}=0 $. Then, whenever $\nabla f \neq 0 $, we have
	\begin{equation}\label{eq:W-ineq}
		\begin{split}
			R^{2} |W^{\pm}|^{2}
			- 36R \det W^{\pm}
			+ 4 |W^{\pm}|^{2} \big|\mathring{{Rc}}\big|^{2} 
			- R \big\langle (\mathring{{Rc}}\circ \mathring{{Rc}})^{\pm},
			W^{\pm}\big\rangle \geq 0.
		\end{split}
	\end{equation}
	Let $\lambda_{1},\lambda_{2},\lambda_{3},\lambda_{4}$ be the eigenvalues of $ \mathring{Rc}$, with
	corresponding orthonormal eigenvectors $\{e_{1} = {\nabla f}/{|\nabla f|}, e_{2}, e_{3}, e_{4}\}$. Then, the equality in \eqref{eq:W-ineq} holds if and only if one of the following occurs:
	\begin{itemize}
		\item[(1)] $W^{\pm}=0$; or
		\item[(2)] after a possible permutation of $ \{e_{2},e_{3},e_{4}\} $,
		\begin{equation*}
			\lambda_{1} = \lambda_{2} = -\lambda,	\qquad
			\lambda_{3} = \lambda_{4} = \lambda,	\qquad
			R = 4\lambda,
		\end{equation*}
		for some $\lambda$.
	\end{itemize}
\end{lem}

We are now ready to establish our classification of four-dimensional asymptotically cylindrical gradient steady Ricci solitons with half-harmonic Weyl curvature.

\begin{proof} [\bf Proof of Theorem \ref{thm:cylin_half-harmonic}]
	First, note that the asymptotically cylindrical assumption (see \cite[Definition on p. 191]{Brendle:14}) implies $R>0$.  

	Without loss of generality, we assume $\delta W^{+}=0$. Then, using Chen--Wang \cite[Theorem 1.1]{Chen-Wang:15}, the proof reduces to establishing the following claim:

	\medskip
	\noindent {\bf Claim 1.}  $W^+\equiv0$ on $M^4$; that is, $(M^4, g, f)$ is half locally conformally flat. 
	\medskip

	\noindent {\bf Proof of Claim 1.} We argue by contradiction. Suppose  that $W^+$ does not vanish identically on $M^4$. Then, whenever  $|W^+| \neq 0$, by Lemmas \ref{lem:WWW-Delta_F} and \ref{lem:WWW>=0}, we have
	\begin{equation} \label{eq:W^+/R}
		\begin{split}
			\Delta_{F} \left(\frac{|W^{+}|}{R}\right)
			&\geq 
			\frac{1}{2 |W^{+}|\,R^{2}}\Big(
			R^{2} |W^{+}|^{2}
			- 36R \det W^{+}
			+ 4 |W^{+}|^{2} \big|\mathring{{Rc}}\big|^{2} \\
			&\quad - R \big\langle (\mathring{{Rc}}\circ \mathring{{Rc}})^{+},
			W^{+}\big\rangle
			\Big) \\
			&\geq 0.
		\end{split}
	\end{equation}
	
	On the other hand, since $M^4$ is asymptotically cylindrical, we have
	\begin{equation*}
		\bar{W} \equiv 0 \quad\&\quad \bar{R}>0,
	\end{equation*}
	where bars denote corresponding curvature quantities on the asymptotic shrinking cylinders $(\rS^{n-1}\times \mathbb{R}, \bar{g}(t))$. In particular, ${|\bar{W}^+|}/{\bar{R}} \equiv 0$.

	Note that the quantity $|W^+|/R \ge 0$ on $M^4$ is scaling invariant. Hence, it attains its maximum at some interior point $p_0 \in M^4$; see also the proof of \cite[Lemma 5.1]{Cao-Xie:25b}. By \eqref{eq:W^+/R} and Calabi's barrier strong maximum principle (see, e.g., \cite[Lemma 2.2]{Cao-Xie:25b}), it follows that $|W^+|/R$ is a constant and
	\begin{equation*}
		R^{2} |W^{+}|^{2}
		- 36R \det W^{+}
		+ 4 |W^{+}|^{2} |\mathring{Rc}|^{2}
		- R \big\langle (\mathring{Rc} \circ \mathring{Rc})^{+}, W^{+} \big\rangle
		\equiv 0
	\end{equation*}
	in a neighborhood $\Omega$ of $p_0$.
	
	Since gradient Ricci solitons are real analytic, the functions $|\nabla f|^{2}$, $|W^{\pm}|^{2}$, and $R$ are real analytic on $M^4$. Therefore, by analyticity and the equality characterization in Lemma~\ref{lem:WWW>=0}, exactly one of the following possible cases holds on $\Omega$ and hence on all of $M^4$: either $\nabla f \equiv 0$, or the second equality case in Lemma~\ref{lem:WWW>=0} holds identically.
	
	\smallskip
	{\bf Case 1.} If $\nabla f \equiv 0$ on $M^{4}$, then $(M^{4},g,f)$ is Ricci-flat, contradicting $R>0$.
	
	\smallskip
	{\bf Case 2.} If $\nabla f \not\equiv 0$ (and $W^{\pm} \not\equiv 0$), then the second equality case in Lemma~\ref{lem:WWW>=0} holds on an open dense subset $U \subset M$. In this case, by Lemma~\ref{lem:WWW>=0}~(2), we have
	\begin{equation*}
		R_{11} = \mathring{R}_{11} + \tfrac{R}{4} = -\lambda+\lambda =0.
	\end{equation*}
	Since $\nabla f$ is an eigenvector of $Rc$ \cite[Lemma 2.4]{WWW:18} (see also Lemma~\ref{lem:WWW>=0}), it follows that 
	\begin{equation*}
		\nabla R = 2 Rc (\nabla f)
		= 2 R_{11} |\nabla f|\, e_{1}
		= 0
	\end{equation*}
	on $U$, and therefore $R$ is constant on $U$. By continuity, $R$ is constant on all of $M^4$. On the other hand, with $\rho=0$ for the steady soliton, we have	
	\begin{equation*}
		0 \equiv \Delta_f R = 2\rho R - 2 |Rc|^2 = -2 |Rc|^2,
	\end{equation*}
	contradicting the assumption that $(M^4,g,f)$ has positive scalar curvature.

	This proves {\bf Claim 1}, and Theorem~\ref{thm:cylin_half-harmonic}  follows. 
\end{proof}

The proof of Theorem~\ref{thm:conical_half-harmonic} is similar; for the reader's convenience, we sketch it below.

\begin{proof} [\bf Proof of Theorem \ref{thm:conical_half-harmonic}]
	First, observe that the asymptotic cone of $(M^4,g,f)$ is a non-flat Euclidean cone and hence has positive scalar curvature. It then follows from the strong maximum principle together with  \cite[Theorem~1.6]{Chan:23} that $(M^4,g,f)$ itself has positive scalar curvature, i.e., $R>0$ on $M^4$.
	
	Next, by an argument analogous to that used in the proof of Theorem~\ref{thm:cylin_half-harmonic}, together with \cite[Lemma A.2]{Cao-Xie:25b}, exactly one of the following holds on all of  $M^4$: either $\nabla f \equiv 0$, or $W^{+} \equiv 0$, or the second equality case in Lemma~\ref{lem:WWW>=0} holds identically.
	
	\smallskip
	{\bf Case 1.} If $\nabla f \equiv 0$ on $M^4$, then $(M^{4},g,f)$ is Einstein with negative Einstein constant, contradicting the fact that $(M^4, g)$ has positive scalar curvature.
	
	\smallskip
	{\bf Case 2.} If $W^{+} \equiv 0$ on $M^4$, then by \cite[Corollary 1.1 \& Remark~1.3]{Cao-Xie:25b}, $(M^4,g,f)$ has positive curvature operator and is rotationally symmetric.
	
	\smallskip
	{\bf Case 3.} If $\nabla f \not\equiv 0$ and $W^{\pm} \not\equiv 0$, then, an argument similar to that in Case~2 of the proof of Theorem~\ref{thm:cylin_half-harmonic} shows that $R$ is constant on all of $M^4$, hence
	\begin{equation*}
		0 \equiv \Delta_f R = 2\rho R - 2 |Rc|^2,
	\end{equation*}
	where $\rho<0$, again contradicting  $R>0$ on $M^4$. 

	Therefore, only {\bf Case 2} can occur, and Theorem~\ref{thm:conical_half-harmonic} follows.
\end{proof}

\begin{rmk} \label{rmk:half-harmonic_ALCF}
	By examining the conditions at infinity used in the proofs of Theorems \ref{thm:cylin_half-harmonic} and \ref{thm:conical_half-harmonic}, we see that the asymptotic assumptions in Theorem \ref{thm:cylin_half-harmonic} and in Theorem \ref{thm:conical_half-harmonic} can be replaced by the following weaker condition of {\em asymptotically half locally conformally flat}:   
\end{rmk}

\begin{defn} \label{defn:ALCF}
	A four-manifold $(M^4,g)$ is said to be asymptotically half locally conformally flat if it has positive scalar curvature $R>0$ and, for any sequence of points $\{p_i\} \subset M^4$, with $p_i \rightarrow \infty$,  satisfies
	\begin{equation*}
		\lim_{i\rightarrow \infty} \frac{|W^{+}|}{R} (p_i)= 0 \quad ({\rm or}\  \lim_{i\rightarrow \infty} \frac{|W^{-}|}{R} (p_i)= 0), 
	\end{equation*}
	where  $W^{+}$ (or $W^{-}$) denotes the self-dual (or the anti-self-dual) part of the Weyl tensor.
\end{defn}

\section{Four-dimensional Ricci expanders with half-WPIC \& half-PIC} \label{sec:half-WPIC}

In this section, we study the classification of four-dimensional gradient  expanding solitons with half-WPIC and half-PIC. In particular, we prove Theorem \ref{thm:half-WPIC-expander} and obtain partial classification results for expanding and steady solitons with half-PIC under additional conditions; see Theorems \ref{thm:conical_half-PIC} and \ref{thm:cylin_half-PIC}.

For the reader's convenience, we restate Theorem \ref{thm:half-WPIC-expander} with a more explicit formulation of part (iii).

\begin{thm}  \label{thm:halfWPIC}
	Let $(M^4, g, f)$ be a four-dimensional complete noncompact, non-flat, gradient expanding Ricci soliton with half-nonnegative isotropic curvature. Then, either
	\begin{itemize}
		\item[(i)] $(M^4, g, f)$ has half-positive isotropic curvature, or
		\smallskip
		
		\item[(ii)] $(M^4, g, f)$ is a locally irreducible expanding K\"ahler-Ricci soliton, or
		\smallskip
		
		\item[(iii)] $(M^4, g, f)$ is isometric to a quotient of either $N^3\times \R$, where $N^3$ is a $3$-dimensional expanding Ricci soliton, or $(\R^2,g_6(\nu)) \times \R^2$, or $(\R^2,g_6(\nu)) \times (\R^2,g_6(\nu))$, or $(\R^2,g_6(\nu)) \times (\R^2,g_7(\nu))$, or $(\R^2,g_6(\nu)) \times (\R_{\ast}^2,g_8(\nu))$, where $(\R^2,g_6(\nu))$, $(\R^2,g_7(\nu))$, and $(\R_{\ast}^2,g_8(\nu))$ are the two-dimensional complete gradient expanding Ricci solitons in \cite[Theorem 1]{Bernstein-Mettler:15}; see also \cite{Ramos:18}.
	\end{itemize}
\end{thm}

\begin{proof} 
	It is well known that the half-WPIC condition implies that the scalar curvature satisfies $R \geq 0$.  
	Following Micallef--Wang \cite{Micallef-Wang:93}, we consider the quantity
	\begin{equation*}
		P^{\pm} := \tfrac{R}{6}I - W^{\pm}.
	\end{equation*}
	From the curvature decomposition \eqref{eq:Rmdecomp}, it is easy to see that 
	\[
	P^{+} = \tfrac{R}{4}I - A, 
	\qquad 
	P^{-} = \tfrac{R}{4}I - C,
	\qquad 
	\tr(P^{\pm}) = \tfrac{R}{2}.
	\]
	Moreover, the half-WPIC (respectively, half-PIC) is equivalent to requiring either $P^{+} \geq 0$ (respectively, $P^{+} > 0$) or $P^{-} \geq 0$ (respectively, $P^{-} > 0$).  
	
	Without loss of generality, we assume $P^{+} \geq 0$ and denote its  eigenvalues by $0\leq P^{+}_1 \leq P^{+}_2 \leq P^{+}_3$. Next, consider the canonical Ricci flow
	$$g(t)=(1+t)\Phi(t)^{\ast}(g), \quad t\in [0,1], \quad g(0)=g,$$
 	induced by the expanding soliton $(M^4,g,f)$.
	By the proof of \cite[Proposition 4.6]{Micallef-Wang:93},
	there exists a positive constant $0<\delta<1$ such that, for any fixed $\tau \in (0,\delta)$, $\ker (P^{+}(\tau ))$ with respect to $g(\tau)$ is invariant under parallel translation and constant in time. It follows that the rank of $P^{+}= P^{+}(0)$ is locally constant, and $\ker(P^{+})$ is invariant under parallel translation with respect to $g$.
	
	Now,  the key step in the proof is to establish the following claim:
	\medskip

	\noindent{\bf Claim 2.} If the holonomy group $\text{Hol}^{0} (M^4, g)$ is $\text{SO}(4)$, then $(M^4, g)$ has half-PIC. 

	\medskip
	\noindent {\bf Proof of Claim 2.}
	We argue by contradiction, following the proof of \cite[Theorem 1.2]{Cao-Xie:23}. Suppose that there exist a point $p\in M^4$ and a non-zero $\vphi_1^{+} \in \wedge^{+}_p(M)$ such that
		$$ P^{+} (\vphi_1^{+},\vphi_1^{+}) = 0. $$ 
	Since  $P^{+} \geq 0$, it follows that $\vphi_1^{+}$ is a null eigenvector corresponding to the smallest eigenvalue $P_1^{+}=0$ at $p$. By \cite[Lemma 6.1]{Derdzinski:00}, we may write
	\begin{equation*}
		\vphi_1^{+} = \tfrac{1}{\sqrt{2}}\left( e_1\wedge e_2 \pm e_3\wedge e_4\right)
	\end{equation*}
	for some positively oriented orthonormal frame $\{e_1, e_2, e_3, e_4\}$.

	Now, let $\vphi_3^{+}\in \wedge^{+}_p(M)$ be an eigenvector corresponding to the largest eigenvalue $P_3^{+}$.  Again by \cite[Lemma 6.1]{Derdzinski:00},  there exists  another positively oriented orthonormal frame $\{v_1, v_2, v_3, v_4\}$ such that
	\begin{equation*}
		\vphi_3^{+} = \tfrac{1}{\sqrt{2}}\left( v_1\wedge v_2 \pm v_3\wedge v_4\right). 
	\end{equation*}		
	Since $\text{Hol}^{0} (M^4, g)=\text{SO}(4)$, there exists a closed loop $\gamma$ based at $p$ such that 
	$ v_i = P_{\gamma}e_i$ for $ i=1, \cdots, 4,$
	where $P_{\gamma}$ denotes parallel transport along $\gamma$. Using the invariance of $\ker(P^{+})$, we obtain
	$$P_3^{+}=P^{+}(\vphi_3^{+},\vphi_3^{+}) = P^{+}(\vphi_1^{+},\vphi_1^{+})=P_1^{+}=0.$$ 
	Hence, 
	$$R=2(P_1^{+}+P_2^{+}+P_3^{+})=0 \quad {\rm at} \ p.$$ 

	Since $R\geq 0$ on $M^4$, it attains its minimum value $0$ at $p\in M^4$. By the proof of \cite[Proposition 3.2]{Petersen-W:09}, $(M^4,g,f)$ is Ricci-flat. Then \cite[Proposition 3.1]{Petersen-W:09} implies that $(M^4,g,f)$ is flat, contradicting the non-flatness assumption of $(M^4,g,f)$. Therefore, $P^{+}>0$ and {\bf Claim 2} follows. 
	
	\smallskip
	The remainder of the proof proceeds as in that of \cite[Theorem~1.4]{Cao-Xie:25}.
	
	\smallskip
	{\bf Case 1:} $(M^4,g,f)$ is locally reducible.  

	Since $(M^4, g, f)$ is non-flat, it must be a finite quotient of either $N^3\times \R$, or $(\R^2,g_6(\nu)) \times \R^2$, or $(\R^2,g_6(\nu)) \times (\R^2,g_6(\nu))$, or $(\R^2,g_6(\nu)) \times (\R^2,g_7(\nu))$, or $(\R^2,g_6(\nu)) \times (\R_{\ast}^2,g_8(\nu))$. Here, $N^3$ is  a three-dimensional expanding soliton, while $(\mathbb{R}^2,g_6(\nu))$, $(\mathbb{R}^2,g_7(\nu))$, and $(\mathbb{R}^2_\ast,g_8(\nu))$ are the two-dimensional complete gradient expanding solitons described in \cite[Theorem~1]{Bernstein-Mettler:15}.
	
	\smallskip
	{\bf Case 2:} $(M^4,g,f)$ is locally irreducible and symmetric. 

	Then $(M^4,g)$ is  Einstein with negative scalar curvature, contradicting $R\geq 0$ (due to half-WPIC). Hence  this case cannot occur. 
	
	\smallskip
	{\bf Case 3:} $(M^4,g,f)$ is locally irreducible and non-symmetric.  

	By Berger's holonomy classification, one of the following holds: 

	\begin{itemize}
		\item[(3a)] 
		If $\text{Hol}^{0} (M^4, g)=\text{SO}(4)$ then by {\bf Claim 2}, $(M^4,g,f)$  has half-PIC.  

		\smallskip
		\item[(3b)] 
		If $\text{Hol}^{0} (M^4, g)=\text{U}(2)$, then $(M^4,g, f)$ is K\"ahler. 

		\smallskip
		\item[(3c)] If $\text{Hol}^{0} (M^4, g)=\text{SU}(2)$, then $(M^4,g,f)$ is Calabi-Yau and hence Ricci-flat. Since Ricci-flat expanding solitons are flat (\cite[Proposition 3.1]{Petersen-W:09}), this contradicts the non-flatness assumption.
	\end{itemize}	

	This completes the proof of Theorem \ref{thm:halfWPIC}.
\end{proof}

\begin{rmk}
	An alternative proof can be obtained using Wilking's extension \cite[Theorem~A.1]{Wilking:13} of the Bony-type maximum principle of Brendle--Schoen \cite{Brendle-S:08}. 
\end{rmk}

\begin{proof} [{\bf Alternative proof of Claim 2.}]
	Without loss of generality, we assume that $(M^4,g)$ has WPIC on self-dual $2$-forms, i.e., $A$ is 2-nonnegative. Following \cite{Richard-Seshadri:16}, we define
	\begin{equation*}
		S := \{ \varphi \in \mathfrak{so}(4,\mathbb{C}) \mid \tr(\varphi^2)=0,\ \varphi \ \text{corresponds to a self-dual $2$-form} \}.
	\end{equation*}
	Then,
	\begin{equation*}
		\mathcal{C}(S) := \{ Rm \in S^2_B(\mathfrak{so}(4,\mathbb{C})) \mid \langle Rm(\varphi),\varphi \rangle \geq 0 \ \text{for all } \varphi \in S \}
	\end{equation*}
	defines a Ricci flow invariant curvature cone of half-WPIC (see \cite[Propositions~2.11 and 2.12]{Richard-Seshadri:16}).  
	By   \cite[Theorem~A.1]{Wilking:13}, and the fact that $(M^4,g,f)$ generates a self-similar solution to the Ricci flow, it follows that $\ker(Rm)\subset S$ is invariant under parallel translation when we restrict $Rm$ to $S$.  
	
	Next, we show that if  $\mathrm{Hol}^{0}(M^4,g)$ is $\mathrm{SO}(4)$, then $(M^4,g,f)$ has half-PIC.  
	We proceed by contradiction. Restrict $Rm$ to $S$ and denote its eigenvalues by $0\leq \alpha_1 \leq \alpha_2 \leq \alpha_3$.  
	Suppose that there exists a point $p \in M^4$ and a nonzero $\varphi_1 \in S$ such that
	\[
	\langle Rm(\varphi_1), \varphi_1 \rangle = 0.
	\]
	Then, it follows that $\varphi_1$ is a null eigenvector of $Rm$ corresponding to the smallest eigenvalue $\alpha_1=0$.  
	By the proof of \cite[Proposition~2.11]{Richard-Seshadri:16}, we can write
	\[
	\varphi_1 = (e_1 + i e_2)\wedge(e_3 + i e_4)
	\]
	for some positively oriented orthonormal frame $\{e_1,e_2,e_3,e_4\}$.
	
	Now let $\varphi_3 \in S$ be an eigenvector corresponding to the largest eigenvalue $\alpha_3$ at $p$.  
	Again by the proof of \cite[Proposition~2.11]{Richard-Seshadri:16}, there exists another positively oriented orthonormal frame $\{v_1,v_2,v_3,v_4\}$ such that
	\[
	\varphi_3 = (v_1+i v_2)\wedge(v_3+i v_4).
	\]
	Since $\mathrm{Hol}^{0}(M^4,g)=\mathrm{SO}(4)$, there exists a closed loop $\gamma$ based at $p$ such that 
	\[
	v_i = P_{\gamma} e_i, \quad i=1,\cdots,4,
	\]
	where $P_{\gamma}$ denotes parallel transport along $\gamma$.  
	It then follows from the invariance of $\ker(Rm)\subset S$ that
	\[
	\alpha_3 = \langle Rm(\varphi_3),\varphi_3 \rangle 
	= \langle Rm(\varphi_1),\varphi_1 \rangle 
	= \alpha_1 = 0.
	\]
	Hence, at $p$, 
	\[
	R = 2(\alpha_1+\alpha_2+\alpha_3) = 0,
	\]
	a contradiction. 

	Thus, $(M^4,g,f)$ has half-PIC.
\end{proof}

\begin{rmk}
	The above two methods can also be applied to classify four-dimensional complete gradient shrinking and steady Ricci solitons with half-WPIC previously established in \cite[Theorem 1.2]{Cao-Xie:23} and \cite[Theorem 1.3]{Cao-Xie:25}.
\end{rmk}

Next, we consider a special class of gradient expanding solitons with half-PIC.

\begin{thm} \label{thm:conical_half-PIC}
	Let $(M^4,g,f)$ be a four-dimensional complete, noncompact, asymptotically conical gradient expanding Ricci soliton with half-positive isotropic curvature. If $M^4$ is asymptotic to a non-flat Euclidean cone and its Ricci tensor has an eigenvalue with multiplicity three, then it has positive curvature operator and is rotationally symmetric.
\end{thm}

\begin{proof}
	Under the curvature operator decomposition \eqref{eq:Rmdecomp}, let 
	$$A_1 \leq A_2 \leq A_3, \qquad C_1 \leq C_2 \leq C_3$$ 
	be the eigenvalues of $A$ and $C$ respectively, and let $0 \leq B_1 \leq B_2 \leq B_3$ be the singular eigenvalues of $B$. Without loss of generality, we assume $C$ is 2-positive, i.e., $C_1+C_2>0$. In particular, $C_2>0$ on $M^4$. 
	
	As in the proof of \cite[Theorem 1.3]{Cao-Xie:23},  the assumption that the Ricci tensor has an eigenvalue of multiplicity three implies
	\begin{equation} \label{eq:B1=B2=B3_conical}
		B_1 = B_2 = B_3.
	\end{equation}
	Moreover, it follows from the proof of \cite[Remark 4.2]{Cao-Xie:23} that
	\begin{equation*}
		\begin{split}
			\Delta_F \frac{C_3 - C_1}{R}
			&\geq
			\frac{2}{R^2}
			\left[
			(C_3 - C_1)\, |\mathring{Rc}|^2
			+ 3R(C_3 - C_1)C_2
			\right]
			\geq 0.
		\end{split}
	\end{equation*}
	
	On the other hand, since the asymptotic cone $\mathcal{C}$ is a non-flat Euclidean cone, \cite[Equation (A.9)]{Cao-Xie:25b} and \cite[Equation (A.10)]{Cao-Xie:25b} imply that 
	\begin{equation*}
		\bar{A}_i = \bar{C}_j = \bar{B}_k, \quad 1 \le i,j,k \le 3,
	\end{equation*}
	where bars denote quantities on the asymptotic cone $\mathcal{C}$. In particular,
	\begin{equation*}
			\frac{\bar{C}_3 - \bar{C}_1}{\bar{R}} = 0.
	\end{equation*}
	Since $(C_3 - C_1)/R \ge 0$ on $M^4$, it attains its maximum at some interior point $p_0 \in M^4$.
	By \eqref{eq:B1=B2=B3_conical} and Calabi's barrier strong maximum principle \cite{Calabi:58},  it follows that $(C_3 - C_1)/R$ is constant and
	\begin{equation*}
		(C_3 - C_1) |\mathring{Rc}|^2 + 3R(C_3 - C_1) C_2 \equiv 0
	\end{equation*}
	in a neighborhood $\Omega$ of $p_0$. 

	Thus, either $C_3 = C_1$, or $\mathring{Rc} = 0$ and $C_2 = 0$ on $\Omega$.  The latter is ruled out by the half-PIC assumption. Hence $C_3 = C_1$ on $\Omega$, which implies $W^- \equiv 0$ on $\Omega$, and therefore on all of $M^4$ by analyticity. 

	Hence, $(M^4, g, f)$ is half locally conformally flat. By \cite[Corollary 1.1 \& Remark 1.3]{Cao-Xie:25b}, it has positive curvature operator and is rotationally symmetric.
\end{proof}

Finally, using an analogous argument, we also obtain a partial classification of asymptotically cylindrical gradient steady solitons with half-PIC.

\begin{thm} \label{thm:cylin_half-PIC}
	Let $(M^4,g,f)$ be a four-dimensional complete, noncompact, asymptotically cylindrical gradient steady Ricci soliton with half-positive isotropic curvature. If the Ricci tensor of $M^4$ has an eigenvalue with multiplicity three, then it is rotationally symmetric and hence isometric to the Bryant soliton.
\end{thm}

\begin{proof} [Sketch  of Proof]
	Arguing as in the proof of Theorem \ref{thm:conical_half-PIC} and using the fact that the scaling invariant quantity $(C_3 - C_1)/R = 0$ on the round cylinder, we conclude that $(M^4,g,f)$ is half locally conformally flat. Therefore, by \cite[Theorem 1.1]{Chen-Wang:15}, $(M^4, g, f)$ is either Ricci-flat or isometric to the Bryant soliton up to scalings. The Ricci-flat case is excluded by the half-PIC condition, hence $(M^4, g, f)$ is isometric to the Bryant soliton up to scalings.
\end{proof}

\begin{rmk}
	As in Remark \ref{rmk:half-harmonic_ALCF},  the asymptotic assumptions in Theorems \ref{thm:conical_half-PIC} and \ref{thm:cylin_half-PIC} can be replaced by the condition of {\em asymptotically half locally conformally flat} given in Definition \ref{defn:ALCF}. Indeed, this condition implies that ${|W^{\pm}|}/{R} = 0$ at infinity, which in turn yields $$\frac{C_3-C_1}{R} = 0, \qquad {\rm or} \qquad  \frac{A_3-A_1}{R} = 0,$$ at infinity, since $A= W^+ + ({R}/{12})I$ and $C = W^- + ({R}/{12})I$. 
\end{rmk}

\medskip


\begin{thebibliography}{99} 

	\bibitem{Bernstein-Mettler:15}
	Bernstein, J.; Mettler, T., {\em Two-dimensional gradient Ricci solitons revisited}, Int. Math. Res. Not. IMRN 2015, no.~1, 78--98.
	
	\bibitem{Brendle:13}
	Brendle, S., {\em Rotational symmetry of self-similar solutions to the Ricci flow}, Invent. Math. {\bf 194} (2013), no.~3, 731--764.
	
	\bibitem{Brendle:14}
	Brendle, S., {\em Rotational symmetry of Ricci solitons in higher dimensions}, J. Differential Geom. {\bf 97} (2014), no.~2, 191--214.

	\bibitem{Brendle-S:08} 
	Brendle, S.; Schoen, R., {\em Classification of manifolds with weakly 1/4-pinched curvatures}, Acta Math. {\bf 200} (2008), no.~1, 1--13.

	\bibitem{Calabi:58}
	Calabi, E., {\em An extension of E. Hopf's maximum principle with an application to Riemannian geometry}, Duke Math. J. {\bf 25} (1958), 45--56.
	
	\bibitem{Cao:97}
	Cao, H.-D., {\em Limits of solutions to the K\"ahler-Ricci flow}, J. Differential Geom. {\bf 45} (1997), no.~2, 257--272.
	
	\bibitem{CaoCCMM:14}
	Cao, H.-D.; Catino, G.; Chen, Q.; Mantegazza, C.; Mazzieri, L., {\em Bach-flat gradient steady Ricci solitons}, Calc. Var. Partial Differential Equations {\bf 49} (2014), no.~1-2, 125--138.
	
	\bibitem{Cao-Chen:12}
	Cao, H.-D.; Chen, Q., {\em On locally conformally flat gradient steady solitons}, Trans. Amer. Math. Soc. {\bf 364} (2012), no.~5, 2377--2391.
	
	\bibitem{Cao-Xie:23}
	Cao, H.-D.; Xie, J., {\em Four-dimensional complete gradient shrinking Ricci solitons with half positive isotropic curvature}, Math. Z. {\bf 305} (2023), no.~2, Paper No.~25, 22 pp.
	
	\bibitem{Cao-Xie:25}
	Cao, H.-D.; Xie, J., {\em Four-dimensional gradient Ricci solitons with (half) nonnegative isotropic curvature}, J. Math. Pures Appl. (9) {\bf 197} (2025), Paper No.~103686, 21 pp.
	
	\bibitem{Cao-Xie:25b}
	Cao, H.-D.; Xie, J., {\em Curvature pinching of asymptotically conical gradient expanding Ricci solitons}, preprint (2025), arXiv:2510.05075.
	
	\bibitem{Cao-Yu:21}
	Cao, H.-D.; Yu, J., {\em On complete gradient steady Ricci solitons with vanishing $D$-tensor}, Proc. Amer. Math. Soc. {\bf 149} (2021), no.~4, 1733--1742.
	
	\bibitem{Catino-Mantegazza:11}
	Catino, G.; Mantegazza, C., {\em Evolution of the Weyl tensor under the Ricci flow}, Ann. Inst. Fourier (Grenoble) {\bf 61} (2011), no.~4, 1407--1435.

	\bibitem{Chan:23}
	Chan, P.-Y., {\em Curvature estimates and gap theorems for expanding Ricci solitons}, Int. Math. Res. Not. IMRN 2023, no.~1, 406--454.

	\bibitem{ChenBL:09}
	Chen, B.-L., {\em Strong uniqueness of the Ricci flow}, J. Differential Geom. {\bf 82} (2009), no.~2, 363--382.
	
	\bibitem{Chen-Zhu:00}
	Chen, B.-L.; Zhu, X.-P., {\em Complete Riemannian manifolds with pointwise pinched curvature}, Invent. Math. {\bf 140} (2000), no.~2, 423--452.

	\bibitem{Chen-Wang:15}
	Chen, X.; Wang, Y., {\em On four-dimensional anti-self-dual gradient Ricci solitons}, J. Geom. Anal. {\bf 25} (2015), no.~2, 1335--1343. 

	\bibitem{Cho-Li:23}
	Cho, J.~H.; Li, Y., {\em Ancient solutions to the Ricci flow with isotropic curvature conditions}, Math. Ann. {\bf 387} (2023), no.~1-2, 1009--1041.

	\bibitem{Deng-Zhu:20a}
	Deng, Y.; Zhu, X., {\em Higher dimensional steady Ricci solitons with linear curvature decay}, J. Eur. Math. Soc. (JEMS) {\bf 22} (2020), no.~12, 4097--4120.
	
	\bibitem{Deng-Zhu:20b}
	Deng, Y.; Zhu, X., {\em Classification of gradient steady Ricci solitons with linear curvature decay}, Sci. China Math. {\bf 63} (2020), no.~1, 135--154.
	
	\bibitem{Derdzinski:00} 
	Derdzinski, A., {\em Einstein metrics in dimension four}, Handbook of differential geometry, Vol. I, 419--707, North-Holland, Amsterdam, 2000. 
	
	\bibitem{Enders-Mueller-Topping:11}
	Enders, J.; M\"uller, R.; Topping, P., {\em On type-I singularities in Ricci flow}, Comm. Anal. Geom. {\bf 19} (2011), no.~5, 905--922.
	
	\bibitem{Ha:82}
	Hamilton, R., {\em Three-manifolds with positive Ricci curvature}, J. Differential Geom. {\bf 17} (1982), no.~2, 255--306.
	
	\bibitem{Ha:88}
	Hamilton, R., {\em The Ricci flow on surfaces}, Contemp. Math. {\bf 71} (1988), 237--261.
	
	\bibitem{Ha:95F}
	Hamilton, R., {\em The formation of singularities in the Ricci flow}, Surveys in Differential Geometry (Cambridge, MA, 1993), {\bf 2}, 7--136, Int. Press, Cambridge, MA, 1995.
	
	\bibitem{Ha:97}
	Hamilton, R., {\em Four-manifolds with positive isotropic curvature}, Comm. Anal. Geom. {\bf 5} (1997), no.~1, 1--92.
	
	\bibitem{Kim:17}
	Kim, J., {\em On a classification of 4-d gradient Ricci solitons with harmonic Weyl curvature}, J. Geom. Anal. {\bf 27} (2017), no.~2, 986--1012.
	
	\bibitem{Kim:25}
	Kim, J., {\em Classification of gradient Ricci solitons with harmonic Weyl curvature}, J. Geom. Anal. {\bf 35} (2025), no.~5, Paper No. 139, 33 pp.
	
	\bibitem{Lai:24}
	Lai, Y., {\em A family of 3D steady gradient solitons that are flying wings}, J. Differential Geom. {\bf 126} (2024), no.~1, 297--328.
	
	\bibitem{Lai:25}
	Lai, Y., {\em $O(2)$-symmetry of 3D steady gradient Ricci solitons}, Geom. Topol. {\bf 29} (2025), no.~2, 687--789.
	
	\bibitem{Lai:25b}
	Lai, Y., {\em 3D flying wings for any asymptotic cones}, J. Differential Geom. {\bf 130} (2025), no.~3, 677--695.
	
	\bibitem{LiFJ:25}
	Li, F., {\em Rigidity of complete gradient steady Ricci solitons with harmonic Weyl curvature}, Pacific J. Math. {\bf 335} (2025), no.~2, 323--353.

	\bibitem{Li-Ni-Wang:18} 
	Li, X.; Ni, L.; Wang, K., {\em Four-dimensional gradient shrinking solitons with positive isotropic curvature}, Int. Math. Res. Not. IMRN 2018, no. 3, 949--959.
	
	\bibitem{Micallef-Wang:93} 
	Micallef, M. J.; Wang, M. Y., {\em Metrics with nonnegative isotropic curvature}, Duke Math. J. {\bf 72} (1993), no.~3, 649--672.
	
	\bibitem{Naber:10}
	Naber, A., {\em Non-compact shrinking four solitons with nonnegative curvature}, J. Reine Angew. Math. {\bf 645} (2010), 125--153.

	\bibitem{Perelman:03} 
	Perelman, G., {\em Ricci flow with surgery on three-manifolds}, arXiv:math/0303109. 
	
	\bibitem{Petersen-W:09} 
	Petersen, P.; Wylie, W., {\em Rigidity of gradient Ricci solitons}, Pacific J. Math. {\bf 241} (2009), no.~2, 329--345.
	
	\bibitem{Ramos:18} 
	Ramos, D., {\em Ricci flow on cone surfaces}, Port. Math. {\bf 75} (2018), no.~1, 11--65.
	
	\bibitem{Richard-Seshadri:16} 
	Richard, T.; Seshadri, H., {\em Positive isotropic curvature and self-duality in dimension 4}, Manuscripta Math. {\bf 149} (2016), no.~3-4, 443--457.
	
	\bibitem{Wilking:13} 
	Wilking, B., {\em A Lie algebraic approach to Ricci flow invariant curvature conditions and Harnack inequalities}, J. Reine Angew. Math. {\bf 679} (2013), 223--247.
	
	\bibitem{WWW:18}
	Wu, J.-Y.; Wu, P.; Wylie, W., {\em Gradient shrinking Ricci solitons of half harmonic Weyl curvature}, Calc. Var. Partial Differential Equations {\bf 57} (2018), no.~5, Paper No.~141, 15~pp.
	
\end{thebibliography}
\end{document}